\documentclass[letter, 11pt]{article}
\usepackage[english]{babel}
\usepackage[utf8]{inputenc}

\usepackage[margin = 1.2 in, top = 1 in, bottom = 1.2 in]{geometry}
\makeatletter
\g@addto@macro\normalsize{%
  \setlength\abovedisplayskip{7pt}
  \setlength\belowdisplayskip{7pt}
  \setlength\abovedisplayshortskip{7pt}
  \setlength\belowdisplayshortskip{7pt}
}
\makeatother

\usepackage{graphicx}
\usepackage{cancel}

\usepackage{enumerate}
\usepackage{scalerel,stackengine}
\usepackage{graphics}
\usepackage{fancyhdr}
\usepackage{cancel}
\usepackage{enumerate}
\usepackage{amssymb}
\usepackage{amsmath}
\usepackage{hyperref}
\usepackage{mdwlist}
\usepackage{etoolbox}
\usepackage{latexsym}
\usepackage{amsthm}
\usepackage{multicol}
\usepackage{makeidx}
\usepackage{bm}
\usepackage{dsfont}
\usepackage{graphicx}
\usepackage{wrapfig}
\usepackage{multicol}
\usepackage{xcolor}
\usepackage{makeidx}
\usepackage{bm}
\usepackage{dsfont}
\usepackage{mdframed,color}
\usepackage{amsmath,amssymb}
\usepackage[dvipsnames]{xcolor}
\usepackage{mathtools}

\usepackage[capitalize]{cleveref}
\usepackage{bbm, dsfont}
\usepackage{tikz}
\usetikzlibrary{matrix}
\usetikzlibrary{positioning}

\usepackage{titlesec}
\titleformat{\section}
  {\large\center\bfseries}
  {\thesection.}{.7em}{}
\titlespacing*{\section}{0pt}{3.5ex plus 0ex minus 0ex}{1.5ex plus 0ex}
\titleformat{\subsection}
  {\center\bfseries}
  {\thesubsection.}{.7em}{}
\titlespacing*{\subsection}{0pt}{3.5ex plus 0ex minus 0ex}{1.5ex plus 0ex}
\titleformat{\subsubsection}
  {\center\bfseries}
  {\thesubsubsection.}{.7em}{}
\titlespacing*{\subsubsection}{0pt}{3.5ex plus 0ex minus 0ex}{1.5ex plus 0ex}

\addto\captionsenglish{}

\usepackage{titling}
\newtheoremstyle{plain}{3mm}{3mm}{\slshape}{}{\bfseries}{.}{.5em}{}
\newtheoremstyle{definition}{2mm}{2mm}{}{}{\bfseries}{.}{.5em}{}
\theoremstyle{plain} 
	
\newtheorem{theorem}{Theorem}[section]

\newtheorem{proposition}[theorem]{Proposition}

\newtheorem{lemma}[theorem]{Lemma}

\theoremstyle{definition} 

\theoremstyle{plain} 
\newcounter{MainTheoremCounter}

\AddToHook{env/problem/begin}{\crefalias{theorem}{problem}}
\AddToHook{env/proposition/begin}{\crefalias{theorem}{proposition}}
\AddToHook{env/question/begin}{\crefalias{theorem}{question}}
\AddToHook{env/conjecture/begin}{\crefalias{theorem}{conjecture}}
\AddToHook{env/lemma/begin}{\crefalias{theorem}{lemma}}
\AddToHook{env/corollary/begin}{\crefalias{theorem}{corollary}}
\AddToHook{env/definition/begin}{\crefalias{theorem}{definition}}
\AddToHook{env/remark/begin}{\crefalias{theorem}{remark}}
\AddToHook{env/example/begin}{\crefalias{theorem}{example}}

\AddToHook{env/maintheorem/begin}{%
  \crefalias{MainTheoremCounter}{theorem}%
}

\theoremstyle{plain}
\newtheorem*{namedthm}{\namedthmname}
\newcounter{namedthm}
	

\crefname{theorem}{Theorem}{Theorems}
\Crefname{theorem}{Theorem}{Theorems}

\crefname{proposition}{Proposition}{Propositions}
\Crefname{proposition}{Proposition}{Propositions}

\crefname{lemma}{Lemma}{Lemmas}
\Crefname{lemma}{Lemma}{Lemmas}

\crefname{corollary}{Corollary}{Corollaries}
\Crefname{corollary}{Corollary}{Corollaries}

\crefname{conjecture}{Conjecture}{Conjectures}
\Crefname{conjecture}{Conjecture}{Conjectures}

\crefname{question}{Question}{Questions}
\Crefname{question}{Question}{Questions}

\crefname{definition}{Definition}{Definitions}
\Crefname{definition}{Definition}{Definitions}

\crefname{remark}{Remark}{Remarks}
\Crefname{remark}{Remark}{Remarks}

\crefname{example}{Example}{Examples}
\Crefname{example}{Example}{Examples}

\crefname{maintheorem}{Theorem}{Theorems}
\Crefname{maintheorem}{Theorem}{Theorems}

\crefname{maincorollary}{Corollary}{Corollaries}
\Crefname{maincorollary}{Corollary}{Corollaries}

\crefname{mainquestion}{Question}{Questions}
\Crefname{mainquestion}{Question}{Questions}

\newcommand{\N}{\mathbb{N}}
\newcommand{\Z}{\mathbb{Z}}
\newcommand{\C}{\mathbb{C}}

\newcommand{\sA}{\mathcal{A}}

\newcommand{\norm}[1]{\left\lVert#1\right\rVert}

\providecommand{\norm}[1]{\lVert #1\rVert}

\usepackage[normalem]{ulem}

\newcommand{\g}{{\mathbf{g}}}
\newcommand{\h}{{\mathbf{h}}}

\newcommand{\cor}{{\mathcal{C}}}
\newcommand{\nil}{{\mathcal{N}}}
\newcommand{\nul}{{\mathbf{e}}}

\makeatletter
\def\thanks#1{%
  \begingroup
  \let\@thefnmark\relax
  \@footnotetext{#1}%
  \endgroup
}

\begin{document}
\author{By~~{\scshape Jovan Andreevski}~~and~~{\scshape Felipe Hern\'andez} }
\date{\small \today}
\thanks{The second author acknowledges OpenAI’s ChatGPT for Academic Researchers program for providing access to ChatGPT during the preparation of this work.}
\title{{\bfseries On the structure of polynomial multicorrelation sequences arising from nilpotent actions}}

\maketitle

\begin{abstract}
We prove that every polynomial multicorrelation sequence arising from nilpotent actions admits a decomposition into a generalized nilsequence and a null-sequence. This resolves a long-standing question of Frantzikinakis concerning the structure of polynomial multicorrelation sequences of commuting transformations. Our proof builds upon ideas underlying the norm convergence theorem of Walsh.
\end{abstract}

\section{Introduction}
A multicorrelation sequence is a sequence of the form
\begin{equation}
\label{eq:proposal-multicorrelation}
    n\mapsto \int f_0 \cdot \left( \prod_{i=1}^{d} T_i^{a_{i,1}(n)} \right) f_1 \cdot \ldots \cdot  \left( \prod_{i=1}^{d} T_i^{a_{i,k}(n)} \right) f_k d\mu
\end{equation}
where $(X,\mu, T_1, \ldots, T_d)$ is a $\Z^d$-measure-preserving system, $f_0,\ldots,f_k\in L^\infty(\mu)$ are bounded functions, and $a_{i,j}\colon\mathbb N\to\mathbb Z$ are sequences for $i=1,\ldots,d$ and $j=1,\ldots,k$. We refer the reader to \cref{preliminaries} for all the necessary background in ergodic theory. Multicorrelation sequences are central objects in ergodic theory. Their systematic study originates in Furstenberg's ergodic proof of Szemer\'edi's theorem \cite{Furstenberg77}. Since then, a series of papers 
\cite{ACKELSBERG_BERGELSON_2024,
Bergelson_Host_Kra05, DonosoFerreMoraguesKoutsogiannisSun2024,
DKKST25, Frantzikinakis2015, Frantzikinakis_Kra_averages_indep_poly:2006,
frantzikinakis2025ergodicaveragessparsecorners,FrantzikinakisKuca2025,    Koutsogiannis18,koutsogiannis2026structure2stepnilpotentergodic,KoutsogiannisLeMoreiraRichter21, 
LE_2020, LeMoreiraRichter2021, Leibman10,Leibman15,leng2025structuredextensions, Moragues2020PropertiesOM,  Moreira_Richter19} has further developed our understanding of multicorrelation sequences, producing important applications in combinatorics, recurrence, and the theory of multiple ergodic averages. The common factor in these results is the search for understanding of the intrinsic structure of multicorrelation sequences, that is, a decomposition into a main component, whose algebraic structure we understand, and an asymptotically negligible error term. The first fundamental result towards this direction was due to Bergelson, Host, and Kra
\cite[Theorem~1.9]{Bergelson_Host_Kra05} establishing that, whenever $(X,\mathcal B,\mu,T)$ is ergodic, 
\begin{equation}
\label{eq:proposal-bhk}
    \int_X f_0\cdot T^n f_1\cdot \ldots \cdot  T^{dn}f_d\,d\mu
    =\nil(n)+\nul(n), \quad n\in \N
\end{equation}
where $\nil$ is a $d$-step generalized nilsequence and $\nul$ is a null-sequence. Here, a \textit{generalized nilsequence} $\nil$ is a uniform limit of nilsequences. A \textit{nilsequence} $\psi:\Z\to \C$ is a sequence of the form $\psi=F(g^n\Gamma)$ where $F:G/\Gamma\to \C$ is a smooth function, $g\in G$, and $G/\Gamma$ is a \textit{nilmanifold}, meaning that $G$ is a nilpotent Lie group and $\Gamma$ is a discrete cocompact subgroup of $G$. We say that $\psi$ has degree $\ell$ if $G$ is at most $\ell$-step nilpotent. Accordingly, we say that a generalized nilsequence $\nil$ has degree $\ell$ if it is the uniform limit of $\ell$-step nilsequences. For $\varepsilon\geq 0$, an $\varepsilon$\textit{-null-sequence} $\nul_\varepsilon: \Z\to \C$ is a sequence satisfying
$$\limsup_{N-M\to\infty} \frac{1}{N-M} \sum_{n=M+1}^N |\nul_\varepsilon(n)|^2 \leq \varepsilon.  $$
A \textit{null-sequence} $\nul:\Z\to \C$ is a $0$-null-sequence. 

Leibman extended the decomposition in \cref{eq:proposal-bhk} to integer polynomial iterates \cite{Leibman10}, and subsequently removed the ergodicity assumption \cite{Leibman15}.


 In the setting of commuting transformations, the situation is considerably less clear. Although mean convergence of multiple ergodic averages for commuting transformations was established for linear iterates by Tao \cite{Tao08}, and subsequently for polynomial iterates by Walsh \cite{Walsh12}, these convergence results do not yield a structural description of the associated multicorrelation sequences. A first general decomposition result was proven by Frantzikinakis \cite[Theorem 1.2]{Frantzikinakis2015}. He established that for every $d,k\in \N$, any polynomials $p_{i,j}\in \Z[t]$, $i\in [d]$, $j\in [k]$, there exists $\ell\in \N$ such that for every $\Z^{d}$-measure-preserving system $(X,\mathcal{B},\mu,T_1,\ldots,T_{d})$, functions $f_1,\ldots,f_k\in L^\infty(\mu)$ bounded by $1$, and $\varepsilon>0$ one has the decomposition 
\begin{equation}\label{Frantzikinakis-equation}
  \int \left( \prod_{i=1}^{d} T_i^{p_{i,1}(n)} \right) f_1 \cdot \ldots \cdot  \left( \prod_{i=1}^{d} T_i^{p_{i,k}(n)} \right) f_k d\mu = \nil_{\varepsilon}(n) + \nul_\varepsilon(n),  \quad n\in \N,  
\end{equation}
where $\nil_{\varepsilon}$ is an $\ell$-step generalized nilsequence bounded by $1$, and $\nul_\varepsilon$ is an $\varepsilon$-null-sequence. We refer to a decomposition of the form \cref{Frantzikinakis-equation} as an \textit{approximate decomposition}, whereas a decomposition of the form \cref{eq:proposal-bhk} as an \textit{exact decomposition}.

The approximate decomposition result of Frantzikinakis  for commuting transformations has since been extended in different directions. Frantzikinakis and Host \cite{FrantzikinakisHost2018}
extended this approximation principle to multivariable polynomial multicorrelation sequences, while Koutsogiannis \cite{Koutsogiannis18} obtained corresponding results for integer-part polynomial iterates. Le, Moreira and Richter~\cite{LeMoreiraRichter2021}
proved that the error term in \cref{Frantzikinakis-equation} can also be chosen with arbitrarily small absolute mean error along primes. They also established polynomial extensions of this
approximation result. Related refinements for integer-part polynomial iterates along primes were obtained by Koutsogiannis, Le, Moreira and
Richter \cite{KoutsogiannisLeMoreiraRichter21}.

A natural question is whether \cref{Frantzikinakis-equation} can be turned into an exact decomposition. Some partial results in this direction include the work of Moragues \cite[Theorem 1.4, Theorem 1.5]{Moragues2020PropertiesOM} and Donoso,
Moragues, Koutsogiannis and Sun \cite[Theorem 2.2]{DonosoFerreMoraguesKoutsogiannisSun2024}, in which they obtained exact decompositions for commuting transformations under suitable ergodicity assumptions. Frantzikinakis and Kuca \cite[Theorem~2.17]{FrantzikinakisKuca2025} established exact decompositions for arbitrary commuting transformations whenever the polynomial iterates are pairwise independent. A further breakthrough was achieved by Leng \cite[Theorem~1.6]{leng2025structuredextensions}, who resolved the linear case without any ergodicity assumptions. More recently, Frantzikinakis and Kuca
\cite[Theorem~1.6]{frantzikinakis2025ergodicaveragessparsecorners} obtained an exact decomposition for arbitrary pairs of polynomial iterates of two commuting transformations.

Despite the progress described above, the general polynomial problem for commuting transformations has remained unresolved for three or more transformations. More precisely, it has not been known whether, for arbitrary integer polynomials \(p_1,\ldots,p_d\), the associated multicorrelation sequence admits an exact decomposition into a generalized nilsequence and a null-sequence. This question has been raised repeatedly in the literature 
\cite[Question 1.1]{DonosoFerreMoraguesKoutsogiannisSun2024}, \cite[Problem 20]{Frantzikinakis11}, \cite[Remark after Theorem 1.1]{Frantzikinakis2015}, \cite[Section 2.7]{FrantzikinakisHost2018}, \cite[Remark after Theorem 1.6]{frantzikinakis2025ergodicaveragessparsecorners}, \cite[Page 636]{FrantzikinakisKuca2025}, \cite[Page 398]{Host_Kra_nilpotent_structures_ergodic_theory:2018}, \cite[Question 2]{KoutsogiannisLeMoreiraRichter21}, \cite[Problem 37]{Kuca2026Survey}, \cite[Question 3]{LeMoreiraRichter2021}.  In this paper, we resolve this problem in full generality. We establish an exact decomposition for polynomial multicorrelation sequences associated with an arbitrary finite family of commuting transformations, without imposing any ergodicity assumptions or restrictions on the iterating polynomials. In fact, our methods allow us to prove the main result in the generality of actions arising from nilpotent groups.
\begin{theorem}\label{main-result}
    Let $G$ be a nilpotent group. For every $d, k\in \N$, any polynomials $p_{i,j}\in \Z[t]$, $i\in [d]$, $j\in [k]$, every $G$-measure-preserving system $(X,\mathcal B,\mu,G)$, with the transformations $T_1,\ldots,T_d\in G$, and every $f_0,\ldots,f_k\in L^\infty(\mu)$ bounded by $1$, we have the decomposition
    $$\int f_0\cdot \left( \prod_{i=1}^{d} T_i^{p_{i,1}(n)} \right) f_1 \cdot \ldots \cdot  \left( \prod_{i=1}^{d} T_i^{p_{i,k}(n)} \right) f_k d\mu = \nil(n) + \nul(n),  \quad n\in \N$$
where $\nil$ is a generalized nilsequence bounded by $1$, and $\nul$ is a null-sequence.
\end{theorem}
Our techniques differ from previous exact decomposition results such as \cite{Bergelson_Host_Kra05,Leibman10,Leibman15,leng2025structuredextensions} where a structure theorem for ergodic averages is the main ingredient. In contrast, our work stems from the ideas of Walsh \cite{Walsh12} and can be seen as a natural continuation of Frantzikinakis' argument in \cite{Frantzikinakis2015}. This is the first method yielding an exact decomposition without the use of structure theory or a degree lowering strategy. 

\section*{Acknowledgments}
The authors thank Florian Richter and Ethan Ackelsberg for comments on an early draft of the article.

\noindent \textbf{Use of technology:} No AI tools were used to produce any of the results or the ideas in this article. The only instance of AI use was to create a referee-style report of the paper before submission. For this we used GPT-6 Astra Max, and only minor corrections were found. The authors assume responsibility for all content.

\section{Preliminaries}\label{preliminaries}
\subsection{Basic tools related to multicorrelation sequences}
Let $(X,\mathcal B,\mu)$ be a standard probability space and let $G$ be a group of measure-preserving transformations acting on $X$. We call $(X, \mu, G)$ a $G$-measure-preserving system. In the case $G$ is generated by $d$ commuting transformations $T_1,\ldots,T_d$, we will denote $(X,\mu, T_1, \ldots, T_d)$ and this will be called a $\Z^d$-system. 

Let $G$ be a nilpotent group and let $(X,\mu,G)$ be a $G$-measure-preserving system. The result of Walsh \cite{Walsh12} shows that for any $T_1,\ldots,T_d\in G$, integer-valued polynomials $p_{i, j}$ with $i \in [d]$, $j \in [k]$, and functions $f_1,\ldots,f_k\in L^\infty(\mu)$, the ergodic averages 
$$\frac{1}{N-M} \sum_{n=M+1}^N \prod_{j=1}^k T_1^{p_{1, j}(n)}\cdots T_d^{p_{d, j}(n)}f_j$$
converge in $L^2(\mu)$ as $N-M\to \infty$.

Consider for $j\in [k]$, $g_j(n)= T_1^{p_{1, j}(n)}\cdots T_d^{p_{d, j}(n)}$ with $T_1,\ldots,T_d\in G$ and $p_{i, j}:\Z\to \Z$ polynomials. A (polynomial) multicorrelation sequence in $(X,\mu,G)$ is a sequence $\cor$ of the form 
$$    \cor(n):= \int_{X} f_0\cdot  g_1(n)f_1 \cdot\ldots \cdot  g_k(n)f_k d\mu , \quad n\in \Z, $$
where $f_0,\ldots,f_k\in L^\infty(\mu)$.

We will use the nilpotent version of Frantzikinakis' decomposition result.
\begin{theorem}\label{Frantzikinakis-nilpotent}
    For every $s,d,k\in \N$, any polynomials $p_{i,j}\in \Z[t]$, $i\in [d]$, $j\in [k]$, there is $\ell\in \N$ such that for every measure-preserving system $(X,\mathcal{B},\mu,T_1,\ldots,T_{d})$ with the transformations $T_1,\ldots,T_{d}$ spanning a nilpotent group of class at most $s$, every $f_1,\ldots,f_k\in L^\infty(\mu)$ bounded by $1$, and $\varepsilon>0$ we have the decomposition 
\begin{equation*}
  \int \left( \prod_{i=1}^{d} T_i^{p_{i,1}(n)} \right) f_1 \cdot \ldots \cdot  \left( \prod_{i=1}^{d} T_i^{p_{i,k}(n)} \right) f_k d\mu = \nil_{\varepsilon}(n) + \nul_\varepsilon(n),  \quad n\in \N  
\end{equation*}
where $\nil_{\varepsilon}$ is an $\ell$-step nilsequence bounded by $1$ and $\nul_\varepsilon$ is an $\epsilon$-null-sequence.
\end{theorem}
The proof of \cref{Frantzikinakis-nilpotent} follows from \cite[Theorem 1.2]{Frantzikinakis2015} and \cite[Remark 2.5]{Frantzikinakis2015}. 

We will also need the following proposition, which says that ergodic averages with multicorrelation sequences as weights are uniform averages.
\begin{proposition}\label{Weight-multicor-average}
Let $G$ be a nilpotent group and $(X,\mu,T)$ a $G$-system. Let $g_1,\ldots,g_k,h_1,\ldots,h_l$ be polynomial sequences in $G$. For $u_0,\ldots,u_k\in L^\infty(\mu)$, set
    $$\mathcal{D}(n)=\int u_0 \cdot \prod_{j=1}^k h_j(n)u_j d\mu. $$
    Then, for functions $f_1,\ldots,f_k\in L^\infty(\mu)$, the uniform limit
    \begin{equation}\label{eq-limit-correlation-and-functions}
        \lim_{N-M \to \infty} \frac{1}{N-M} \sum_{n=M+1}^N \mathcal{D}(n) \cdot  \prod_{j=1}^k g_j(n)f_j 
    \end{equation}
    exists in $L^2(\mu) $.
\end{proposition}
For a sequence $a:\Z\to \C$, a F\o lner sequence $\Phi=(\Phi_N)_N$ and $N,M\in \N$ we will denote 
$$\sA_{N,M} a(n):= \frac{1}{|\Phi_N|}\sum_{n\in \Phi_N}  a(n)  - \frac{1}{|\Phi_M|}\sum_{n\in \Phi_M}a(n).     $$
\begin{proof}
    By \cite[Lemma 3.5]{Zorin-Kranich16}, it is enough to prove that for every F\o lner sequence $\Phi=(\Phi_N)_N$, the limit in \cref{eq-limit-correlation-and-functions} for averages along $\Phi$ exists. Fix a F\o lner sequence $\Phi=(\Phi_N)_N$. For $N,M\in \N$, we have that 
    \begin{align*}
        &\norm{\sA_{N,M} \mathcal{D}(n) \cdot  \prod_{j=1}^k g_j(n)f_j    }_{L^2(\mu)}^2 \\
        &= \int \left| \int  u_0(y) \cdot \left(  \sA_{N,M}  \prod_{j=1}^k h_j(n)u_j(y) \cdot  \prod_{j=1}^k g_j(n)f_j(x)   \right) d\mu(y)  \right|^2 d\mu(x)\\
        &\leq \norm{u_0}_{L^2(\mu)}^2 \int  \int  \left|  \sA_{N,M}  \prod_{j=1}^k h_j(n)u_j(y) \cdot  \prod_{j=1}^k g_j(n)f_j(x)   \right|^2  d\mu(y) d\mu(x),
    \end{align*}
    where we used the Cauchy-Schwarz inequality in the last step. We observe that 
    \begin{align*}
         \int  \int  \left|  \sA_{N,M}  \prod_{j=1}^k h_j(n)u_j(y) \cdot  \prod_{j=1}^k g_j(n)f_j(x)   \right|^2  d\mu(y) d\mu(x)&= \norm{A_{N,M}  \prod_{j=1}^k (h_j\times g_j)(n)\left(u_j\otimes f_j\right)}_{L^2(\mu\otimes \mu)}^2
    \end{align*}
    where we denote by $(h_j\times g_j)(n)$ the action of $G\times G$ in the product $X\times X$. Since clearly $h_j\times g_j$ is a polynomial sequence in the nilpotent group $G\times G$, we have by Walsh \cite[Theorem 5.1]{Walsh12} that 
    $$\lim_{\min\{N,M\}\to \infty}\norm{A_{N,M}  \prod_{j=1}^k (h_j\times g_j)(n)(u_j\otimes f_j)}_{L^2(\mu\otimes \mu)}=0, $$
    concluding that 
    $$ N\mapsto \frac{1}{|\Phi_N|} \sum_{n\in \Phi_N} \mathcal{D}(n) \cdot  \prod_{j=1}^k g_j(n)f_j$$
    is a Cauchy sequence in $L^2(\mu)$, and thus it converges.
\end{proof}
The next result is standard and it has been proved in many previous articles \cite{Bergelson_Host_Kra05,Leibman15, leng2025structuredextensions, Moragues2020PropertiesOM}. It states that in order to find a nil plus null decomposition, it is enough to find a nil plus null plus a uniform error decomposition. Even though it was first proved in \cite[Lemma 3.3]{Leibman15}, we will cite it in the form in which it is stated in \cite[Lemma A.7]{leng2025structuredextensions}.
\begin{proposition}\label{nil-plus-uniform-error}
Suppose, for every $\varepsilon>0$, a sequence $a(n)$ can be written as a sum of a generalized nilsequence $a_{\text{nil},\varepsilon}(n)$, a null-sequence $a_{\text{null},\varepsilon}(n)$, and a sequence $a_{\varepsilon}(n)$ which is uniformly bounded by $\varepsilon$. Then, $a(n)$ is the sum of a generalized nilsequence and a null-sequence.
\end{proposition}
We remark that the formulation in \cite[Lemma A.7]{leng2025structuredextensions} fixes the step of the generalized nilsequence. The same proof gives the result as stated in \cref{nil-plus-uniform-error}. Alternatively, the formulation in \cite[Lemma 3.3]{Leibman15} also implies our formulation here.

We will also make use of the following lemma.
\begin{lemma}[{\cite[Lemma 3.2]{Leibman15}}]\label{nil+nul-inequality}
    Let $\nil$ be a generalized nilsequence and $\nul$ be a null-sequence. Then
    $$\norm{\nil}_\infty\leq \norm{\nil+\nul}_{\infty}. $$
\end{lemma}

\subsection{Polynomial sequences in nilpotent groups}
We adopt some of the notation introduced in \cite{Walsh12}. Let $G$ be a nilpotent group of measure preserving transformations of a probability space $(X,\mathcal{B},\mu)$. A $G$-sequence is a sequence $(g(n))_n$ of elements in $G$. A $G$-polynomial sequence is a sequence $(g(n))_n$ of the form
$$g(n)=T_1^{p_1(n)}\cdots T_d^{p_d(n)}, $$
for $T_1,\ldots,T_d\in G$ and $p_1, \ldots, p_d$ being integer-valued polynomials. 
Given two $G$-sequences $g$ and $h$, for each $m\in\N$, we set the $G$-sequence 
$$\langle g | h \rangle_m(n) :=g(n)g(n+m)^{-1}h(n+m). $$
An ordered tuple $\g=(g_1,\ldots,g_k)$ of $G$-sequences will be called a $G$-system. We say that two systems are equivalent if they consist of the same $G$-sequences (allowing repetition). For example, for $G$-sequences $g$ and $h$, the systems $(h,g,g)$, $(g,h)$, $(g,g,g,g,h)$ are equivalent. 

For $m\in \N$ and a $G$-system $\g=(g_1,\ldots,g_k)$ we denote its $m$-reduction by 
$$\g_m^*=(g_1,\ldots,g_{k-1}, \langle g_k | 1_G\rangle_m, \langle g_k | g_1\rangle_m,\ldots,\langle g_k | g_{k-1}\rangle_m ), $$
where we denote by $1_G$ the constant sequence equal to the identity $e_G$ in $G$ for every value $n$. 

For a system $\g$ we use the same complexity as in \cite[Definition 3.1]{Walsh12}, this is: $\g$ has complexity $0$ if it is equivalent to the trivial system $(1_G)$. Recursively, $\g$ has complexity $d\in \N$ if it is not of complexity $d'$ for any $0\leq d'<d$ and it is equivalent to some system $\h$ for which every reduction $\h_m^*$ has complexity $\leq d-1$. 

A key fact that we use, established by Walsh \cite[Theorem 4.2]{Walsh12}, is that every system $\g$ of $G$-polynomial sequences has finite complexity.

\section{The proof of the main result}
Let $G$ be a nilpotent group, $(X,\mu,G)$ be a $G$-measure-preserving system and $p_{i,j}\in \Z[t]$ be polynomials for $i\in [d]$, $j\in [k]$. Pick $T_1,\ldots,T_d\in G$. Consider for each $j\in [k]$, 
$$g_j(n)= T_1^{p_{1,j}(n)}\cdots T_d^{p_{d,j}(n)}.$$
Our goal is to show that for $f_0,\ldots,f_k\in L^\infty(\mu)$ bounded by $1$, the multiple correlation sequence $\cor$ defined by
\begin{equation}\label{multicorrelation-sequence}
    \cor(n)= \int_{X} f_0\cdot  g_1(n)f_1 \cdot\ldots \cdot  g_k(n)f_k d\mu , \quad n\in \Z,
\end{equation}
is equal to a generalized nilsequence plus a null sequence.
We argue by induction on the complexity of $\g = (g_1,\ldots,g_k)$. Observe that this result is trivial for a $G$-system of complexity $0$ as the constant sequence is a nilsequence. Assume that the desired decomposition holds for every polynomial system of strictly smaller complexity. Since equivalent systems give rise to the same class of multicorrelation sequences, we may replace $\g$ by an equivalent system, relabelling the functions accordingly. Thus, by the definition of complexity, we may assume that the complexity of $\g_m^*$ is strictly smaller than the complexity of $\g$ for every $m \in \N$.

A \textit{$\g$-basic function} on $X$ is a function $F$ of the form
\begin{equation}\label{basic-function}
    F:= \lim_{N-M\to \infty} \frac{1}{N-M} \sum_{n=M+1}^N \mathcal{D}(n) \cdot \left( g_k(n)^{-1} v_0  \cdot \prod_{j=1}^{k-1} g_k(n)^{-1}g_j(n) v_j  \right) , 
\end{equation}
where $v_0,\ldots,v_{k-1}\in L^\infty(\mu)$ are bounded by $1$ and $\mathcal{D}(n)$ is a multicorrelation sequence of the form \cref{multicorrelation-sequence}. We observe that the limit in \cref{basic-function} exists always and is uniform by \cref{Weight-multicor-average}.

The idea will be to apply a strategy similar to \cite{Walsh12}: we will split the space $L^2(\mu)$:
$$ L^2(\mu)=\mathcal{B}_\g \oplus \mathcal{B}_\g^\perp,$$
where $\mathcal{B}_\g $ is the closure of the space generated by $\g$-basic functions, and $\mathcal{B}_\g^\perp$ is its orthocomplement. In similar fashion to \cite{8203259}, $\g$-basic functions represent the analogue of reducible functions from \cite{Walsh12}. That is, we will see that replacing $f_k$ by a $\g$-basic function produces a correlation of a sequence of lower complexity, for which we can apply an induction process to deduce the desired decomposition. Meanwhile, replacing $f_k$ by a function in $\mathcal{B}_\g^\perp$ renders $\mathcal{C}$ to a null-sequence.

We begin by noticing that if $f_k\in \mathcal{B}_\g^\perp$, then
\begin{align*}
  &\lim_{N-M\to \infty} \frac{1}{N-M}\sum_{n=M+1}^N |\cor(n)|^2\\
  &=  \lim_{N-M\to \infty} \frac{1}{N-M}\sum_{n=M+1}^N \overline{\cor(n)} \int f_0 \cdot  g_1(n)f_1 \cdot \ldots\cdot g_k(n)f_k d\mu\\
    &=\lim_{N-M\to \infty}    \int f_k \cdot \frac{1}{N-M}\sum_{n=M+1}^N\overline{\cor(n)}  \left( g_k(n)^{-1} f_0  \prod_{j=1}^{k-1} g_k(n)^{-1}g_j(n) f_j \right) d\mu.
\end{align*}
Since the limit
\begin{equation*}
    \lim_{N-M\to \infty} \frac{1}{N-M}\sum_{n=M+1}^N \overline{\cor(n)}  \left( g_k(n)^{-1} f_0  \prod_{j=1}^{k-1} g_k(n)^{-1}g_j(n) f_j \right) 
\end{equation*}
exists uniformly in $L^2(\mu)$ by \cref{Weight-multicor-average}, and it is equal to a $\g$-basic function, we conclude that  
$$  \lim_{N-M\to \infty} \frac{1}{N-M}\sum_{n=M+1}^N |\cor(n)|^2=0. $$
Hence, the result follows in this case as $\cor$ is a null-sequence itself.

Now let us turn to the general case and take $f_k\in L^\infty(\mu)$. Let $\mathbf{P}_{\mathcal{B}_\g}(f_k)$ denote the projection of $f_k$ to $\mathcal{B}_\g$. The sequence
$$n\mapsto \tilde{\nul}(n):=\int f_0 \cdot g_1(n) f_1  \cdot \ldots\cdot  g_{k-1}(n) f_{k-1}\cdot g_k(n)(f_k-\mathbf{P}_{\mathcal{B}_\g}(f_k)) d\mu $$
is a null-sequence. This way, it is enough to prove the result with $\mathbf{P}_{\mathcal{B}_\g}(f_k)$ in place of $f_k$. Notice that $\mathbf{P}_{\mathcal{B}_\g}(f_k)$ may not be bounded by $1$. Nevertheless, for each $\varepsilon>0$, there are $\ell\in \N$, $\alpha_1,\ldots,\alpha_{\ell}\in \C$ and $F_1,\ldots,F_\ell\in L^\infty(\mu)$ basic functions bounded by $1$ such that 
$$\norm{\mathbf{P}_{\mathcal{B}_\g}(f_k)-\sum_{i=1}^{\ell} \alpha_i F_i}_{L^2(\mu)}\leq \varepsilon.$$
We observe that the sequence $(S_\varepsilon(n))_n$ defined by
$$ S_\varepsilon(n):=\int_{X} f_0\cdot  g_1(n)f_1 \cdot\ldots \cdot g_{k-1}(n)f_{k-1}\cdot g_k(n)\left(\mathbf{P}_{\mathcal{B}_\g}(f_k)-\sum_{i=1}^{\ell} \alpha_i F_i\right) d\mu, \quad n\in \N  $$
is such that 
$$\norm{S_\varepsilon}_{\infty}\leq \norm{\mathbf{P}_{\mathcal{B}_\g}(f_k)-\sum_{i=1}^{\ell} \alpha_i F_i}_{L^2(\mu)}\leq \varepsilon.  $$
Thus $S_\varepsilon$ is a sequence uniformly bounded by $\varepsilon$. Since 
$$\mathcal{C}(n)=\tilde{\nul}(n)+ S_\varepsilon(n)+ \sum_{i=1}^{\ell} \alpha_i \int f_0\cdot  g_1(n)f_1 \cdot\ldots \cdot  g_{k-1}(n)f_{k-1}\cdot g_k(n) F_i d\mu, $$
we have that if for each $i\in [\ell]$ we prove the result with $F_i$ in place of $f_k$, we will obtain generalized nilsequences $\nil_i$ bounded by $1$ such that 
$$\mathcal{C}(n)=\tilde{\nul}(n)+ S_\varepsilon(n)+ \sum_{i=1}^{\ell} \alpha_i \nil_i(n). $$
As this holds for every $\epsilon>0$, by \cref{nil-plus-uniform-error} we have that there are a generalized nilsequence $\nil$, and a null-sequence $\nul$ such that 
$$ \mathcal{C}(n)=\nil(n)+ \nul(n).$$
Since $\mathcal{C}$ is bounded by $1$, by \cref{nil+nul-inequality} we have that $\nil$ is also bounded by $1$, which would conclude the proof.

Hence, it is enough to prove the result assuming that $f_k$ is a basic function bounded by $1$. In what follows, we continue to denote the resulting correlation sequence by $\mathcal{C}$. Thus, let us write 
$$f_k=  \lim_{N-M\to \infty} \frac{1}{N-M} \sum_{n=M+1}^N \mathcal{D}(n) \cdot \left( g_k(n)^{-1} v_0  \cdot \prod_{j=1}^{k-1} g_k(n)^{-1}g_j(n) v_j  \right) $$
where $\mathcal{D}$ is a multicorrelation sequence of the form \cref{multicorrelation-sequence} and $v_0, \ldots,v_{k-1}\in L^\infty(\mu)$. Since the previous convergence is uniform, then we have that for $\varepsilon>0$, there exists $M\in \N$ such that for all $n\in \N$, we have
$$\norm{f_k - \frac{1}{M}\sum_{m\in n+[M]}\mathcal{D}(m) \cdot \left( g_k(m)^{-1} v_0  \cdot \prod_{j=1}^{k-1} g_k(m)^{-1}g_j(m) v_j  \right)}_{L^2(\mu)}\leq \varepsilon. $$
This way, defining the sequence $(s_\varepsilon(n))_{n\in \N}$ by
\begin{align*}
    s_\varepsilon(n)&= \int f_0\cdot   g_1(n)f_1 \cdot \ldots\cdot  g_{k-1}(n)f_{k-1}\cdot \\
    &g_k(n)\left(f_k - \frac{1}{M}\sum_{m\in [M]}\mathcal{D}(n+m) \cdot \left( g_k(n+m)^{-1} v_0  \cdot \prod_{j=1}^{k-1} g_k(n+m)^{-1}g_j(n+m) v_j  \right)\right) d\mu 
\end{align*}
we have that 
$$\norm{s_\varepsilon}_\infty \leq \norm{f_k - \frac{1}{M}\sum_{m\in n+[M]}\mathcal{D}(m) \cdot \left( g_k(m)^{-1} v_0  \cdot \prod_{j=1}^{k-1} g_k(m)^{-1}g_j(m) v_j  \right)}_{L^2(\mu)} \leq  \varepsilon. $$
Hence, $s_\varepsilon$ is a sequence uniformly bounded by $\varepsilon$. Replacing $f_k$ by this approximation and reordering the terms, we get that
\begin{align*}
    &\int f_0 \cdot g_1(n)f_1 \cdot \ldots\cdot  g_k(n)f_k d\mu\\
    &=s_\varepsilon(n)+  \frac{1}{M}\sum_{m\in [M]}\mathcal{D}(n+m)  \int f_0  \cdot g_1(n)f_1 \cdot \ldots\cdot  g_{k-1}(n)f_{k-1}\cdot \\
    &g_k(n) \left( g_k(n+m)^{-1} v_0  \cdot \prod_{j=1}^{k-1} g_k(n+m)^{-1}g_j(n+m) v_j \right) d\mu\\
    &= s_\varepsilon(n) +\frac{1}{M}\sum_{m=1}^M \mathcal{D}(n + m)R_m(n) , 
\end{align*}
where we defined for each $n \in \N$, $m\in [M]$:
\begin{align*}
    R_m(n) = \int f_0 \cdot  \left( \prod_{j=1}^{k-1} g_j(n) f_j\right) \cdot \langle g_k | 1_G \rangle_m(n) v_0 \cdot \left( \prod_{j=1}^{k-1}\langle g_k | g_j \rangle_m(n)v_j\right)  d\mu.
\end{align*}
We know that the reduction 
$$\g^*_m=(g_1,\ldots,g_{k-1},\langle g_k | 1_G \rangle_m, \langle g_k | g_1 \rangle_m,\ldots, \langle g_k | g_{k-1} \rangle_m) $$
has strictly less complexity than $(g_1,\ldots,g_k)$. Thus, by the induction hypothesis, we have that for each $m\in \N$, the multicorrelation sequence $R_m$ is equal to the sum of generalized nilsequence $\nil_m$ and a null-sequence $\nul_m$, yielding for $n\in \N$:
$$\int f_0 \cdot g_1(n)f_1 \cdot \ldots\cdot  g_k(n)f_k d\mu=  s_\varepsilon(n)+ \frac{1}{M}\sum_{m=1}^M \mathcal{D}(n + m)\nil_m(n) +  \frac{1}{M}\sum_{m=1}^M \mathcal{D}(n + m)\nul_m(n). $$
On the one hand, since $\mathcal{D}$ is bounded, by Cauchy-Schwarz, the sequence $$n\mapsto \frac{1}{M}\sum_{m=1}^M \mathcal{D}(n + m)\nul_m(n)$$ is a null-sequence. On the other hand, as $\mathcal{D}$ is a multicorrelation sequence, by \cref{Frantzikinakis-nilpotent} we have that there is a generalized nilsequence $\widetilde{\nil}$ and an $\eta$-null-sequence $(a_\eta(n))_n$ such that 
$$\mathcal{D}(n)= \widetilde{\nil}(n) + a_\eta(n), $$
where $\eta:=\varepsilon^2/2.$
Without loss of generality, assume that $M$ is big enough so that for each $n\in \N$:
$$ \frac{1}{M}\sum_{m= 1}^M |a_\eta(n+m)|^2 \leq \varepsilon^2 .  $$
Thus, we have that 
$$ \frac{1}{M}\sum_{m=1}^M \mathcal{D}(n + m)\nil_m(n) =  \frac{1}{M}\sum_{m=1}^M  \widetilde{\nil}(n + m)\nil_m(n) +  \frac{1}{M}\sum_{m=1}^M a_\eta(n+m)\nil_m(n).  $$
Since $\|\nil_m\|_\infty\leq 1$ for every
$m\in\mathbb N$ by the induction hypothesis, we observe that for each $n\in \Z$, $$\left|\frac{1}{M}\sum_{m=1}^M a_\eta(n+m)\nil_m(n) \right|\leq \left(\frac{1}{M}\sum_{m= 1}^M |a_\eta(n+m)|^2 \right)^{1/2} \leq \varepsilon. $$
Thus, the sequence $$n\mapsto \frac{1}{M}\sum_{m=1}^M a_\eta(n+m)\nil_m(n)$$
is uniformly bounded by $\varepsilon$. Since the sequence
$$ n\mapsto \frac{1}{M}\sum_{m=1}^M  \widetilde{\nil}(n + m)\nil_m(n) $$
is clearly a generalized nilsequence bounded by 1, as it is a convex combination of generalized nilsequences, we conclude that
\begin{equation}\label{MCS}
    n\mapsto \int f_0 \cdot g_1(n)f_1 \cdot \ldots\cdot  g_k(n)f_k d\mu
\end{equation}
is the sum of a generalized nilsequence, a null-sequence, and a sequence uniformly bounded by $2\varepsilon$. As this holds for every $\varepsilon>0$, by \cref{nil-plus-uniform-error}, we conclude that the sequence in \cref{MCS} is the sum of a generalized nilsequence and a null-sequence, thus finishing the proof.

\footnotesize{
\bibliographystyle{abbrv}
\bibliography{refs}
}
\bigskip
\noindent
Jovan Andreevski\\
\textsc{{\'E}cole Polytechnique F{\'e}d{\'e}rale de Lausanne} (EPFL)\par\nopagebreak
\noindent
\href{mailto:felipe.hernandezcastro@epfl.ch}
{\texttt{jovan.andreevski@epfl.ch}}

\bigskip
\noindent
Felipe Hernández\\
\textsc{{\'E}cole Polytechnique F{\'e}d{\'e}rale de Lausanne} (EPFL)\par\nopagebreak
\noindent
\href{mailto:felipe.hernandezcastro@epfl.ch}
{\texttt{felipe.hernandezcastro@epfl.ch}}

\end{document}